\documentclass[10pt, twoside]{amsart}
\usepackage{amsthm}
\usepackage{amsmath}
\usepackage{amssymb}
\usepackage{amscd}
\usepackage{mathtools}
\usepackage[all]{xy}
\usepackage{indentfirst}
\usepackage{comment}
\usepackage{microtype}
\usepackage{tikz}
\usepackage{tikz-cd}
\usetikzlibrary{patterns}
\usepackage{enumitem}

\usepackage[hidelinks, bookmarks=false]{hyperref}

\newtheorem{thm}{\bf Theorem}[section]
\newtheorem{prop}[thm]{\bf Proposition}
\newtheorem{lem}[thm]{\bf Lemma}

\newtheorem{cor}[thm]{\bf Corollary}

\newtheorem{q}[thm]{\bf Question}

\newtheorem*{thm*}{\bf Theorem}
\newtheorem*{cor*}{\bf Corollary}

\theoremstyle{definition}

\newtheorem{rem}[thm]{\it Remark}

\newtheorem*{df*}{\bf Definition}
\newtheorem*{not*}{\bf Notation}
\newtheorem*{dfs*}{\bf Definitions}
\newtheorem*{ack*}{\bf Acknowledgements}
\newtheorem*{dfrem*}{\bf Definition and Remark}
\newtheorem*{conv*}{\bf Convention}

\def\P{\mathbb{P}}
\def\C{\mathbb{C}}
\def\F{\mathbb{F}}

\def\Q{\mathbb{Q}}

\def\Z{\mathbb{Z}}

\def\E{\mathcal{E}}

\def\X{\mathcal{X}}

\def\O{\mathcal{O}}
\def\PP{\mathbb{P}}

\DeclareMathOperator{\Spec}{Spec}
\DeclareMathOperator{\Proj}{Proj}

\DeclareMathOperator{\Ker}{Ker}

\DeclareMathOperator{\Pic}{Pic}

\DeclareMathOperator{\rank}{rank}

\DeclareMathOperator{\CH}{CH}

\DeclareMathOperator{\Image}{Im}

\DeclareMathOperator{\Coker}{Coker}

\DeclareMathOperator{\alg}{alg}
\DeclareMathOperator{\cl}{cl}
\DeclareMathOperator{\Zar}{Zar}

\DeclareMathOperator{\Top}{top}

\DeclareMathOperator{\NS}{NS}
\DeclareMathOperator{\Gal}{Gal}
\DeclareMathOperator{\van}{van}
\DeclareMathOperator{\nr}{nr}
\DeclareMathOperator{\Gr}{Gr}

\DeclareMathOperator{\et}{\text{\'et}}

\DeclareMathOperator{\Td}{Td}
\DeclareMathOperator{\Cores}{Cores}
\DeclareMathOperator{\inv}{inv}
\makeatletter
\def\myrightarrow{{\setbox\z@\hbox{$\rightarrow$}\dimen0\ht\z@\multiply\dimen0 6\divide\dimen0 10\ht\z@\dimen0\box\z@}}
\def\myrightarrowfill@{\arrowfill@\relbar\relbar\myrightarrow}
\newcommand{\myxrightarrow}[2][]{\ext@arrow 0359\myrightarrowfill@{#1}{#2}}
\makeatother
\newcommand{\isoto}{\myxrightarrow{\,\sim\,}}

\makeatletter
\let\@wraptoccontribs\wraptoccontribs
\makeatother

\subjclass[2020]{14C25, 14C30, 14G12, 14J28}
\keywords{Chow groups, Hasse principle, Hodge classes, Abel-Jacobi maps, Enriques surfaces}
\title[\tiny An $\O$-acyclic variety of even index]
{An $\O$-acyclic variety of even index}
\date{\today}
\author{John Christian Ottem}
\address{Department of Mathematics, University of Oslo, Box 1053, Blindern, 0316 Oslo, Norway}
\email{johnco@math.uio.no}
\author{Fumiaki Suzuki}
\address{UCLA Mathematics Department, Box 951555, Los Angeles, CA, 90095-1555}
\email{suzuki@math.ucla.edu}
\contrib[with an appendix by]{Olivier Wittenberg}
\address{Laboratoire Analyse, G\'eom\'etrie et Applications Institut Galil\'ee -- Universit\'e Sorbonne Paris Nord, 99 avenue Jean-Baptiste Cl\'ement, 93430 Villetaneuse, France}
\email{wittenberg@math.univ-paris13.fr}

\begin{document}

\begin{abstract}
We give the first examples of $\O$-acyclic smooth projective geometrically connected varieties over the function field of a complex curve, whose index is not equal to one.
More precisely, we construct a family of Enriques surfaces over $\P^{1}$ such that any multi-section has even degree over the base $\P^{1}$
and show moreover that we can find such a family defined over $\Q$.
This answers affirmatively a question of Colliot-Th\'el\`ene and Voisin.
Furthermore,
our construction provides counterexamples to: the failure of the Hasse principle accounted for by the reciprocity obstruction; the integral Hodge conjecture; and universality of Abel--Jacobi maps.
\end{abstract}
\maketitle

\section{Introduction}
In a letter to Grothendieck \cite[p. 152]{CS}, Serre asked whether a smooth projective geometrically connected variety $Y$ over the function field of a complex curve should always have a rational point if it is {\it $\O$-acyclic}, 
that is, $H^{i}(Y,\O_{Y})=0$ for all $i>0$.
This indeed holds for rationally connected varieties, as proved by Graber--Harris--Starr \cite{GHS}, generalizing a classical theorem of Tsen. However, Graber--Harris--Mazur--Starr \cite{GHMS} gave a counterexample for the general case; 
in fact, they showed that there exist Enriques surfaces with no rational points over the function field of a complex curve.
Later, more explicit constructions of such Enriques surfaces were given by Lafon \cite{L} and Starr \cite{S}.
It is remarkable that the example of Lafon is defined over $\Q(t)$ and has no rational point over the local field $\C((t))$.

In light of these examples, one might still hope that a weaker statement could be true.
We recall that the {\it index} of a proper variety $Y$ over a field $F$ is defined to be
\[
I(Y)=\gcd\left\{\deg_{F}(\alpha)\mid \alpha \in CH_{0}(Y)\right\}.
\]

One can then ask:
\begin{q}\label{q}
Does an $\O$-acyclic smooth projective geometrically connected variety $Y$ over the function field of a complex curve always have $I(Y)=1$? 
\end{q}
In other words, we ask whether Serre's question has a positive answer if we replace a rational point on $Y$ with a $0$-cycle of degree $1$.

It is important to note that there is no local obstruction here:
the Riemann--Roch theorem implies that $Y$ as in Question \ref{q} always has 
indices one everywhere locally, or equivalently,
that $Y$ gives a one-parameter family $X\rightarrow C$ of $\O$-acyclic varieties with no multiple fiber (see also \cite[Proposition 7.3]{CTV} and \cite[Theorem 1]{ELW}). 

Nevertheless, it was expected by several mathematicians that Question \ref{q} would have a negative answer 
(see \cite{S} for expectations of Esnault on the indices of the examples
of Graber--Harris--Mazur--Starr and Lafon).
In particular, Colliot-Th\'el\`ene and Voisin asked \cite[Question 7.9]{CTV} whether one can construct an $\O$-acyclic surface of index not equal to one.
The aim of this paper is to give the first counterexamples to Question \ref{q} and thereby to answer affirmatively the question raised by Colliot-Th\'el\`ene and Voisin.
Our main result is the following:

\begin{thm}[=Theorem \ref{t1}, \ref{t1Q}]\label{T1}
Let $X\subset \P^{1}\times \P^{2}\times \P^{2}$ be the rank one degeneracy locus of a map of vector bundles
\[
\O^{\oplus 3}\rightarrow \O(2,2,0)\oplus \O(2,0,2).
\]
If $X$ is very general, then the first projection gives a family $X\rightarrow \P^{1}$ of Enriques surfaces such that any multi-section has even degree over the base $\P^{1}$.
That is, the index $I(X_{\eta})$ is even, where $X_{\eta}$ is the generic fiber. Moreover, we can find threefolds with these properties defined over $\Q$.
\end{thm}
\begin{rem}
Our construction can be generalized to give a counterexample to Question \ref{q} when $\dim Y=2n$ for any positive integer $n$
(besides ones obtained from Theorem \ref{T1} by taking the product with a projective space); see Theorem \ref{t1'}.
\end{rem}
\begin{rem}
It would be natural to ask an analogue of Question \ref{q} over the function field of a curve over the algebraic closure of a finite field.
We will prove some conditional positive results in Proposition \ref{Tate} and Corollary \ref{Tate'}.
\end{rem}

Our construction has consequences for certain questions in number theory. 
We say that the Hasse principle holds for $0$-cycles of degree $1$ on a smooth projective geometrically connected variety $Y$ over the function field $F=\C(C)$ of a complex curve $C$ if there is a $0$-cycle of degree $1$ on $Y$ whenever there is such a cycle on $Y_{F_{p}}$ for any point $p\in C$, where $F_{p}\cong \C((t))$ is the completion of $F$ at $p$.
The reciprocity obstruction to the Hasse principle for $0$-cycles of degree $1$ on a variety over the function field of a complex curve, which is an analogue of the Brauer-Manin obstruction for rational points on a variety over a number field, was defined and pointed out to the authors by Colliot-Th\'el\`ene (see also \cite[Section 5]{CTG}).

As a consequence of our construction, we prove that the failure of the Hasse principle for $0$-cycles of degree $1$ on an Enriques surface over $\C(\P^{1})$ cannot always be accounted for by the reciprocity obstruction.

\begin{thm}[=Theorem \ref{HPRO}]
Let $X_{\eta}$ be the generic fiber of a very general family $X\rightarrow \P^{1}$ of Enriques surfaces as in Theorem \ref{T1}.
Then the Hasse principle fails for $0$-cycles of degree $1$ on $X_{\eta}$, while there is no reciprocity obstruction for $X_{\eta}$.
\end{thm}

Question \ref{q} is also related to the integral Hodge conjecture.
We recall that the integral Hodge conjecture in degree $2i$ on a smooth complex projective variety $X$ 
is the statement that degree $2i$ integral Hodge classes on $X$ are algebraic,
i.e., the image $H^{2i}_{\alg}(X,\Z)\subseteq H^{2i}(X,\Z)$ of the cycle class map $\cl^{i}\colon CH^{i}(X)\rightarrow H^{2i}(X,\Z)$ generates the entire group $Hdg^{2i}(X,\Z)=H^{i,i}(X)\cap H^{2i}(X,\Z)$ of integral Hodge classes.
While the statement holds for $i=0,1,\dim X$, it is known to fail in general for $2\leq i\leq \dim X-1$.
The first counterexample was constructed by Atiyah--Hirzebruch \cite{AH} and many others have been found since then \cite{BCC,BO,CTV, OS, Sch, SV, T}.

As pointed out by Colliot-Th\'el\`ene and Voisin \cite[Theorem 7.6]{CTV}, a counterexample to Question \ref{q} gives 
a one-parameter family $X\rightarrow C$ of $\O$-acyclic varieties for which the integral Hodge conjecture fails in degree $2d-2$, where $d=\dim X$.
This means that the defect of the integral Hodge conjecture in degree $2d-2$, defined as
\[
Z^{2d-2}(X)=Hdg^{2d-2}(X,\Z)/H^{2d-2}_{\alg}(X,\Z),
\]
is non-zero.


It follows that the integral Hodge conjecture fails in degree 4 for the threefold $X$ in Theorem \ref{T1}, and that the defect $Z^4(X)$ is non-zero.
In the last part of the paper, we determine completely the $2$-torsion subgroup $Z^{4}(X)[2]$.
In addition, this allows us to compute explicitly the degree $3$ unramified cohomology group $H^{3}_{\nr}(X,\Z/2)$,
a stable birational invariant of smooth complex projective varieties defined in the framework of the Bloch--Ogus theory \cite{BlO}.
A key input is a theorem of Colliot-Th\'el\`ene and Voisin \cite[Theorem 3.9]{CTV} together with the fact that we have $CH_{0}(X)=\Z$ (this can be deduced from a result of Bloch--Kas--Lieberman \cite{BKL}).

\begin{thm}[=Theorem \ref{t2}, Corollary \ref{co1}]\label{T2}
Let $X$ be the total space of a very general family of Enriques surfaces as in Theorem \ref{T1}.
Then we have 
\[H^{3}_{\nr}(X,\Z/2)=Z^{4}(X)[2]=(\Z/2)^{46}.\]
\end{thm}
\begin{rem}
Note that there is a $2$-torsion element in the N\'eron-Severi group of the geometric generic fiber of the family $X\rightarrow \P^{1}$. In contrast, Colliot-Th\'el\`ene and Voisin  proved that if $X\rightarrow C$  is a family of $\O$-acyclic surfaces such that the geometric generic fiber has torsion free N\'eron--Severi group, then the degree $3$ unramified cohomology group with torsion coefficients is conjecturally of rank at most one \cite[Theorem 7.7, 8.21, Remark 8.22]{CTV}.
\end{rem}

We note that Theorem \ref{T1} also has an application to universality of the Abel-Jacobi maps.
A classical question of Murre asks whether the Abel-Jacobi map is universal among all regular homomorphisms (see \cite[Section 4]{OS} and \cite[Section 1]{Su} for more precise statements).
Recently, a negative answer to the question was given by a fourfold constructed by the authors \cite{OS}.
In fact, the threefold $X$ of Theorem \ref{T1} can be used to construct another such fourfold.
We refer the reader to the papers \cite{OS} and \cite{Su} for the details of the argument.

This paper is organized as follows.
In Section \ref{FES}, we introduce certain families of Enriques surfaces parametrized by $\P^{1}$ and study their basic properties. 
In Section \ref{PMT}, we prove the main theorem over $\C$ using an explicit geometric construction.
The proof involves a combination of monodromy and specialization arguments, and a key congruence obtained previously  by the authors in \cite{OS}.
In Section \ref{DEQ}, we refine this construction to get counterexamples defined over $\Q$.
In Section \ref{FHPRO}, we discuss the failure of the Hasse principle and the reciprocity obstruction on our examples.
In Section \ref{DIHCUC}, we compute the defect of the integral Hodge conjecture in degree $4$ on the total space of the family of Enriques surfaces of the main theorem, 
and in addition, its degree $3$ unramified cohomology group with $\Z/2$ coefficients.
Finally, in the Appendix, Olivier Wittenberg proves that the vanishing of the reciprocity obstruction obtained in Theorem \ref{HPRO} is in fact a completely general phenomenon. 

\medskip 

\noindent {\bf Notation.} We work over the complex numbers in Section \ref{FES}, \ref{PMT}, \ref{FHPRO}, and \ref{DIHCUC}. In Section \ref{DEQ}, we work over $\overline{\Q}$. We use Grothendieck's notation for projective bundles: for a vector bundle $\mathcal E$, $\P(\mathcal E)$ parameterizes one-dimensional quotients of $\mathcal E$. We write $\O_{\PP(\mathcal E)}(1)$ for the relative hyperplane bundle.

We will let $\O_{\PP^r\times \PP^s}(a,b)$ and $\O_{\PP^r\times \PP^s\times \PP^t}(a,b,c)$ denote line bundles on products of projective spaces (i.e., these are $\mbox{pr}_1^*\O_{\PP^r}(a)\otimes \mbox{pr}_2^*\O_{\PP^s}(b)$ and $\mbox{pr}_1^*\O_{\PP^r}(a)\otimes \mbox{pr}_2^*\O_{\PP^s}(b)\otimes \mbox{pr}_3^*\O_{\PP^t}(c)$ respectively). Similarly, we will write $\O_{\PP^1\times \PP(\E)}(a,b)$ for the line bundle $\mbox{pr}_1^*\O_{\PP^1}(a)\otimes \mbox{pr}_2^*\O_{\PP(\E)}(b)$ on $\PP^1\times \PP(\E)$. To simplify notation we will usually drop the subscripts when the context is clear.

\begin{ack*}
We would like to thank Lawrence Ein, J\'anos Koll\'ar, J\o rgen Vold Rennemo, and Burt Totaro for interesting discussions and useful suggestions.
We wish to thank Jean-Louis Colliot-Th\'el\`ene for many helpful correspondences and for encouraging us to check whether there is any reciprocity obstruction on our examples, which led to Theorem \ref{HPRO}.
We are grateful to Olivier Wittenberg for his remarks and for kindly agreeing to write an appendix for our paper. 
Finally, we thank the referee for careful reading and valuable comments.
\end{ack*}

\section{Families of Enriques surfaces parametrized by $\P^{1}$}\label{FES}
We will fix the following notation:
\begin{itemize}[leftmargin=*]
\item $\P_{A}=\P_{\P^{2}\times \P^{2}}(\mathcal{O}(2,0)\oplus \mathcal{O}(0,2))$, $E_{1}=\P_{\P^{2}\times \P^{2}}(\mathcal{O}(2,0))$, $E_{2}=\P_{\P^{2}\times \P^{2}}(\mathcal{O}(0,2))$
\item $\P_{B}=\P_{\P^{2}\times \P^{2}}(\mathcal{O}(1,0)\oplus \mathcal{O}(0,1))$, $F_{1}=\P_{\P^{2}\times \P^{2}}(\mathcal{O}(1,0))$, $F_{2}=\P_{\P^{2}\times \P^{2}}(\mathcal{O}(0,1))$
\item $\P_{C}=\P(H^{0}(\P_{B},\mathcal{O}(1)))$, $P_{1}=\P(H^{0}(\P^{2}\times \P^{2},\mathcal{O}(1,0)))$, $P_{2}=\P(H^{0}(\P^{2}\times \P^{2},\mathcal{O}(0,1)))$.
\end{itemize}

As is explained in \cite{OS}, these spaces are related by the following geometric construction: 
$\P_C$ is a 5-dimensional projective space, and $P_1$ and $P_2$ define disjoint planes in it via the isomorphism $$H^{0}(\P_{B},\mathcal{O}(1))=H^{0}(\P^{2}\times \P^{2},\mathcal{O}(1,0))\oplus H^{0}(\P^{2}\times \P^{2},\mathcal{O}(0,1)).$$ The projective bundle $\P_B$ is then identified with the blow-up of $\P_{C}$ along the union of $P_{1}$ and $P_{2}$, and $F_1$ and $F_2$ are the corresponding exceptional divisors. Furthermore, there is an involution $\iota$ on $\P_{C}$ induced by the involution on $H^{0}(\P_{B},\mathcal{O}(1))$ 
with the ($\pm 1$)-eigenspaces $H^{0}(\P^{2}\times \P^{2},\mathcal{O}(1,0))$ and $H^{0}(\P^{2}\times \P^{2}, \mathcal{O}(0,1))$, respectively.
The involution $\iota$ lifts to an involution on the blow-up $\P_{B}$, and we have $\P_{A}=\P_{B}/\iota$.  Thus there is a double cover $\P_{B}\rightarrow \P_{A}$ over $\P^{2}\times \P^{2}$, which is ramified along $F_1\cup F_2$, and  the divisors $F_{i}$ are mapped isomorphically onto $E_{i}$ for $i=1,2$. 

The varieties $\P_A,\P_B,\P_C$ were used in \cite{OS} to give projective models of Enriques surfaces. 
In this paper, we will use them to study the threefolds $X$ in Theorem \ref{T1}; these are Enriques surface fibrations over $\P^{1}$. 
We now explain the main construction.

Let $X\subset \P^{1}\times \P^{2}\times \P^{2}$ be the rank one degeneracy locus of a general map of vector bundles
\begin{equation}\label{defX}
\O^{\oplus 3}\rightarrow \O(2,2,0)\oplus \O(2,0,2).
\end{equation}
Then $X$ is a smooth threefold and the first projection $X\rightarrow \P^{1}$ defines a family of Enriques surfaces (see \cite[Lemma 2.1]{OS}). 
There is a natural diagram
\[
\xymatrix{
\P^{1}\times \P_{B}\ar[d]\ar[r]&\P^{1}\times \P_{A}\ar[d]\\
\P^{1}\times \P_{C}&\P^{1}\times \P^{2}\times \P^{2}
}
\]
in which $\P^{1}\times \P_{A}\rightarrow \P^{1}\times \P^{2}\times \P^{2}$ is the natural projection; 
$\P^{1}\times \P_{B}\rightarrow \P^{1}\times\P_{A}$ is the quotient map by the involution $\iota$ (which acts trivially on $\P^{1}$); and $\P^{1}\times\P_{B}\rightarrow \P^{1}\times \P_{C}$ is the blow-up along the union of $\P^{1}\times P_{1}$ and $\P^{1}\times P_{2}$.

The above diagram restricts to a diagram
\[
\xymatrix{
Y\ar[d]\ar[r]&X'\ar[d]^{\simeq}\\
Y_{\min}&X
}
\]
where $X'\subset\P^{1}\times \P_{A}$, $Y\subset \P^{1}\times \P_{B}$, and $Y_{\min}\subset \P^{1}\times\P_{C}$ are respectively defined by 
a section of $\O(2,1)^{\oplus 3}$ on $\PP^1\times \P_{A}$ and $\iota$-invariant sections of $\O(2,2)^{\oplus 3}$ on $\PP^1\times \P_{B}$ and $\PP^1\times \P_{C}$ induced by the map of vector bundles defining $X$.

Note that each of the intersections $Y_{\min}\cap (\P^{1}\times P_{i})$ is a complete intersection of three divisors of type $(2,2)$ on $\PP^1\times \PP^2$; thus they consist of $24$ points $y_{i,1},\ldots, y_{i,24}$. Moreover,
 the map $Y\rightarrow Y_{\min}$ is the blow-up of $Y_{\min}$ along the 48 points $y_{i,j}$, with the corresponding exceptional divisors $F_{i,j}$ being the components of $Y\cap (\P^{1}\times F_{i})$. The double cover $Y\rightarrow X'$ is ramified exactly along the union of the $F_{i,j}$, and each $F_{i,j}$ is mapped isomorphically onto $E_{i,j}$ (the components of $X'\cap (\P^{1}\times E_{i})$).
If $X$ is general, the map $\PP^{1}\times\PP_A\to \PP^1\times \PP^2\times \PP^2$ restricts to an isomorphism $X'\rightarrow X$.

\begin{rem}
The minimal model $X_{\min}$ of $X$ can be obtained by contracting the projective planes $E_{i,j}$ to points;
$X_{\min}$ is singular exactly at the images of $E_{i,j}$, and at each of the singular points the tangent cone is the affine cone over a Veronese surface.
\end{rem}

\begin{lem}\label{l1} The threefold $X$ has the following properties:
\begin{enumerate}
\item The degree homomorphism $\deg \colon CH_{0}(X)\rightarrow \Z$ is an isomorphism.
\item The canonical divisor of $X$ is of the form
\[
K_{X}=4F+\frac{1}{2}\sum_{i=1}^{2}\sum_{j=1}^{24} E_{i,j},
\]
where $F$ is the class of a fiber of the projection $X\rightarrow \P^{1}$.
Thus $X$ has Kodaira dimension $\kappa(X)=1$.
\item The topological Euler characteristic equals $\chi_{\Top}(X)=-96$ and Hodge diamond is given by
\[
\xymatrixrowsep{0.01in}
\xymatrixcolsep{0.01in}
\xymatrix{
&&&1&&&\\
&&0&&0&&\\
&0&&50&&0&\\
0&&99&&99&&0\\
&0&&50&&0&\\
&&0&&0&&\\
&&&1&&&
}
\]
\item $X$ is simply connected and the cohomology groups $H^{i}(X,\Z)$ are torsion-free for all $i$.
\end{enumerate}
\end{lem}

\begin{proof}
The arguments are entirely analogous to those in \cite[Section 2]{OS},
where the case of the rank one degeneracy locus of a general map of vector bundles
\[
\O^{\oplus 3}\rightarrow \O(1,2,0)\oplus \O(1,0,2)
\]
is considered.
The properties (1) to (4) correspond to the statements of \cite[Lemma 2.4, 2.5, 2.6, 2.7]{OS} respectively.
We note that the property (1) can be deduced from a result of Bloch--Kas--Lieberman \cite{BKL} on the Chow group of $0$-cycles on an Enriques surface.
\end{proof}

We will also need the following:

\begin{lem}\label{l2}
Let $X$ be the threefold defined by \eqref{defX}.
Then we have
\[
H^{2}(X,\Z)=\frac{\Z[F]\oplus\Z[H_{1}]\oplus\Z[H_{2}]\oplus \bigoplus_{i=1}^{2}\bigoplus_{j=1}^{24}\Z[E_{i,j}]}{\langle-2[H_{1}]+\sum_{j=1}^{24}[E_{1,j}]=-2[H_{2}]+\sum_{j=1}^{24}[E_{2,j}]\rangle)},
\]
where $F$ is the class of a fiber of the first projection $X\rightarrow \P^{1}$ 
and $H_{1}$ (resp. $H_{2}$) is the pullback of the class of a line in $\P^{2}$ via the composition $X\rightarrow \P^{2}\times\P^{2}\xrightarrow{pr_{1}}\P^{2}$ (resp. $X\rightarrow \P^{2}\times\P^{2}\xrightarrow{pr_{2}}\P^{2}$).
\end{lem}
\begin{proof}
Let $X^{\circ}=X \setminus \bigcup_{i=1}^{2}\bigcup_{j=1}^{24}E_{i,j}$.
The long exact sequence for cohomology with supports yields
\begin{equation}\label{extension}
0\rightarrow \bigoplus_{i=1}^{2}\bigoplus_{j=1}^{24}\Z[E_{i,j}]\rightarrow H^{2}(X,\Z)\rightarrow H^{2}(X^{\circ},\Z)\rightarrow 0.
\end{equation}
Let $Y^{\circ}=Y\setminus \bigcup_{i=1}^{2}\bigcup_{j=1}^{24}F_{i,j}$.
Since $X^{\circ}$ is the quotient of $Y^{\circ}$ by the group $\langle\iota\rangle=\Z/2$, which acts freely, we can apply the Cartan--Leray spectral sequence
\[
E^{p,	q}_{2}=H^{p}(\Z/2, H^{q}(Y^{\circ},\Z))\Rightarrow H^{p+q}(X^{\circ},\Z).
\]
We have $H^{1}(Y^{\circ},\Z)=H^{1}(Y_{\min},\Z)=0$ by the Lefschetz hyperplane section theorem, so we have a short exact sequence
\[
0\rightarrow \Z/2\rightarrow H^{2}(X^{\circ},\Z)\rightarrow H^{2}(Y^{\circ},\Z)^{\iota}\rightarrow 0.
\]
The long exact sequence for cohomology with supports yields
\[
0\rightarrow \bigoplus_{i=1}^{2}\bigoplus_{j=1}^{24}\Z[F_{i,j}]\rightarrow H^{2}(Y,\Z)\rightarrow H^{2}(Y^{\circ},\Z)\rightarrow 0.
\]
Applying the Lefschetz hyperplane section theorem to $Y_{\min}$,
it is straightforward to compute
\[
H^{2}(Y,\Z)=\frac{\Z[F]\oplus\Z[H_{1}]\oplus\Z[H_{2}]\oplus \bigoplus_{i=1}^{2}\bigoplus_{j=1}^{24}\Z[F_{i,j}]}{\langle-[H_{1}]+\sum_{j=1}^{24}[F_{1,j}]=-[H_{2}]+\sum_{j=1}^{24}[F_{2,j}]\rangle}.
\]
Thus we obtain
\[
H^{2}(Y^{\circ},\Z)=\frac{\Z[F]\oplus\Z[H_{1}]\oplus \Z[H_{2}]}{\langle[H_{1}]=[H_{2}]\rangle}
\]
and $H^{2}(Y^{\circ},\Z)$ is $\iota$-invariant.
This, combined with the equality
\[-2[H_{1}]+\sum_{j=1}^{24}[E_{1,j}]=-2[H_{2}]+\sum_{j=1}^{24}[E_{2,j}]\]
 in $H^{2}(X,\Z)$, implies that
\[
H^{2}(X^{\circ},\Z)=\frac{\Z[F]\oplus \Z[H_{1}]\oplus\Z[H_{2}]}{\langle2[H_{1}]=2[H_{2}]\rangle},
\]
and the claim follows immediately from \eqref{extension}.
\end{proof}

\section{Proof of the main theorem}\label{PMT}
We will now prove Theorem \ref{T1} over the complex numbers. 

\begin{thm}\label{t1}
Let $X\subset \P^{1}\times \P^{2}\times \P^{2}$ be the rank one degeneracy locus of a very general map of vector bundles
\[
\O^{\oplus 3}\rightarrow \O(2,2,0)\oplus \O(2,0,2).
\]Then the first projection gives a family $X\rightarrow \P^{1}$ of Enriques surfaces such that any multi-section has even degree over the base $\P^{1}$.
That is, the index $I(X_{\eta})$ is even, where $X_{\eta}$ is the generic fiber.
\end{thm}

\begin{proof}

The first goal will be to prove that for any $1$-cycle $\alpha$ on $X$ and for any $12$-tuple of integers $1\leq j_{1}<\cdots<j_{12}\leq 24$, there is a congruence
\begin{eqnarray}\label{c2}
\deg(\alpha/\P^{1})\equiv \sum_{k=1}^{12}\alpha\cdot E_{1,j_{k}} \mod 2.
\end{eqnarray}
These congruences will imply the theorem. Indeed, from (\ref{c2}) we obtain
\begin{eqnarray}\label{c3}
\alpha\cdot E_{1,1}\equiv \cdots \equiv \alpha\cdot E_{1,24} \mod 2,
\end{eqnarray}
which in turn implies that $\deg(\alpha/\P^{1})$ is even.

To prove the congruence (\ref{c2}), we combine monodromy and specialization arguments.
First, we prove that 
a certain monodromy group acts on the set of $24$ planes $E_{1,1},\ldots, E_{1,24}$ by permutations, and every permutation of the $E_{1,j}$ is realized by this action. 
This will allow us to reduce to proving \eqref{c2} for a fixed 12-tuple $1\le j_1<\cdots<j_{12}\le 24$.

Consider the universal family $$\X\rightarrow G=\Gr(3,H^{0}(\P^{1}\times \P_{A},\O(2,1)))$$ of complete intersections in $\P^{1}\times \P_{A}$ of three divisors of type $(2,1)$. 
Let $\E_{1}$ denote the pullback of the Cartier divisor $E_1$ via the projection map $\X\to \PP_A$. The corresponding family $\E_1\to G$ is the union of the planes $E_{1,1},\ldots,E_{1,24}$ in the fibers of $\X\to G$. Let $\widetilde{G}\rightarrow G$ be the Stein factorization of $\mathcal{E}_{1}\rightarrow G$, which is a finite morphism of degree $24$,
and let $U\subset G$ be the largest open set over which $\widetilde{G}\rightarrow G$ is \'etale. 
We will now prove the following:
\begin{lem}
The monodromy representation
\[
\rho\colon \pi_{1}(U)\rightarrow S_{24},
\]
uniquely determined up to the choice of a base point, is surjective.
\end{lem}
\begin{proof}
Recall from Section \ref{FES} that the planes $E_{1,1},\ldots,E_{1,24}$ are parameterized by the 24 intersection points of three divisors of type $(2,2)$ in $\PP^1\times \PP^2$.  To prove the lemma, we restrict over a certain line $l$ on $G$ defined as follows. Let $\widetilde{l} \subset \P^{1}\times \P^{2}$ be a general complete intersection of two divisors of type $(2,2)$.
Taking a general pencil in $|\O_{\widetilde{l}}(2,2)|$, we obtain a Lefschetz pencil $\widetilde{l}\rightarrow \P^{1}$ by \cite[Theorem XVII. 2.5]{DK}.
This defines a line 
\[l\subset \Gr(3,H^{0}(\P^{1}\times \P^{2}\times \P^{2},\O(2,2,0))) \subset G\]
 such that $\widetilde{l}=l\times_{G}\widetilde{G}$,
 where the inclusion between the Grassmannians is via the identification
 \[
 H^{0}(\P^{1}\times\P_{A},\mathcal{O}(2,1))=H^{0}(\P^{1}\times\P^{2}\times \P^{2},\mathcal{O}(2,2,0))\oplus H^{0}(\P^{1}\times\P^{2}\times \P^{2},\mathcal{O}(2,0,2)).
  \]
We let $l^{\circ}=l\cap U$; this is the maximal open set where $\widetilde{l}\to l$ is \'etale.

We claim that the induced monodromy representation
\[\rho_{l^{\circ}}\colon \pi_{1}(l^{\circ})\rightarrow S_{24}\]
 is surjective. 
 Indeed, 
 $\pi_{1}(l^{\circ})$
 is generated by loops around branch points $b\in  B$ of $\widetilde{l}\to l$, and the image of each such loop is a transposition in $S_{24}$. 
 The image of $\rho_{l^{\circ}}$ is moreover transitive since $\widetilde{l}$ is irreducible. 
 Any transitive subgroup of $S_{24}$ which is generated by transpositions must be $S_{24}$ itself, so it follows that 
 $\rho\colon \pi_{1}(U)\rightarrow S_{24}$ 
 is surjective.
\end{proof}

By the above lemma, we reduce to proving the congruence \eqref{c2} for a single $12$-tuple $1\leq j_{1}< \cdots<j_{12}\leq 24$. 
Indeed, if $g\in \pi_{1}(U)$
is a lift of a permutation $\sigma\in S_{24}$, 
then it will imply for any $1$-cycle $\alpha$ on $X$, there is a congruence
$$
\deg(\alpha/\P^{1})=\deg(g^{*}(\alpha)/\P^{1})\equiv \sum_{k=1}^{12} g^*(\alpha) \cdot E_{j_k} \equiv  \sum_{k=1}^{12} \alpha \cdot E_{\sigma^{-1}(j_k)} \mod 2.
$$
Here we have used $g^*(\alpha)\cdot E_{\sigma(j)} = \alpha \cdot E_j$ for each $j$, and the fact that $g^{*}(\alpha)$ is again an algebraic cycle because $X$ is very general.  We also have $\deg(g^*(\alpha)/\PP^1)=\deg(\alpha/\PP^1)$, because the degree is obtained by intersecting with the class of a fiber over $\PP^1$, which is invariant under monodromy.
Letting $\sigma$ run over all permutations, we see that the congruence will hold for all $12$-tuples.

To finish the proof of the congruence \eqref{c2}, we use a specialization argument. We consider $X$ as the complete intersection of three divisors $D_1,D_2,D_3$ in $|\O(2,1)|$ on $\PP^1\times \PP_A$. 
If we degenerate each $D_i$ to a union $D_i'+D_i''$,
 where $D_i'\in |\O(1,1)|$ and $D_i''\in |\O(1,0)|$ are very general divisors, 
 we obtain a family of threefolds $\X_T\to T$, with special fiber equal to
$$X_{0}\cup R_{1}\cup R_{2}\cup R_{3},$$
where $X_0$ is a very general intersection of three divisors in $|\O(1,1)|$, and $R_{1},R_{2},R_{3}$ are intersections of two divisors of type $\O(1,1)$ and one of type $\O(1,0)$. In particular, the $R_i$ are pairwise disjoint and can be regarded as complete intersections of two relative hyperplane sections in $\PP_A$.
By the geometric construction in Section \ref{FES}, we may regard $X_{0}$ as the rank one degeneracy locus in $\P^{1}\times \P^{2}\times \P^{2}$ of a very general map of vector bundles
\[
\O^{\oplus 3}\rightarrow \O(1,2,0)\oplus \O(1,0,2).
\] 
By construction, $X_{0}$ is  also the only dominant component with respect to the projection $X_{0}\cup R_{1}\cup R_{2}\cup R_{3}\rightarrow \P^{1}$.
Furthermore, again by genericity, we may assume that $X_{0}\cup R_{1}\cup R_{2}\cup R_{3}$ is a simple normal crossing variety and the intersection $(X_{0}\cup R_{1}\cup R_{2}\cup R_{3})\cap (\P^{1}\times E_{1})$ is transversal.

This degeneration allows us to specialize cycles on $X$ to cycles on $X_0\cup R_1\cup R_2 \cup R_2$. On the level of divisors, the union of $24$ components $E_{1,1},\ldots, E_{1,24}$ on $X$ specializes to 
the union of $12$ components $E_{1,1}^{(0)},\ldots, E_{1,12}^{(0)}$ on $X_{0}$ and $4$ components $E_{1,1}^{(l)},\ldots, E_{1,4}^{(l)}$ on $R_{l}$ for $l=1,2,3$ given by the intersections with $\P^{1}\times E_{1}$. 
Thus the chosen specialization gives a $12$-tuple $1\leq j_{1}<\cdots < j_{12}\leq 24$ such that $E_{1,j_{1}},\ldots, E_{1,j_{12}}$ specialize to $E_{1,1}^{(0)},\ldots,E_{1,12}^{(0)}$.

By  \cite[Section 20.3]{F} there is moreover a specialization map of Chow groups
$$
CH_1(X)\to CH_1(X_0\cup R_1\cup R_2 \cup R_3)
$$which is compatible with intersections with Cartier divisors. 
If 
$\alpha_{0}$
is the specialization of 
a $1$-cycle $\alpha$ on $X$,
we may write 
$\alpha_{0}=\alpha_{0}^{(0)}+\alpha_{0}^{(R)}$,
 where 
 $\alpha_{0}^{(0)}$
  is a $1$-cycle on $X_0$ and 
  $\alpha_{0}^{(R)}$ 
   is supported in $R_1\cup R_2\cup R_3$. 

Now we recall a key congruence obtained in the course of the proof of \cite[Theorem 3.1]{OS}: we have, for any $1$-cycle $\alpha_{0}^{(0)}$ on $X_{0}$, a congruence
\begin{eqnarray}\label{c1}
\deg(\alpha_{0}^{(0)}/\P^{1})\equiv \sum_{j=1}^{12}\alpha_{0}^{(0)}\cdot E_{1,j}^{(0)} \mod 2.
\end{eqnarray}
Note that 
$\deg (\alpha/\P^{1})=\deg(\alpha_{0}^{(0)}/\P^{1})$
 and 
 $\alpha \cdot E_{1,j_{k}}=\alpha_{0}\cdot E_{1,k}^{(0)}=\alpha_{0}^{(0)}\cdot E_{1,k}^{(0)}$
 since $E^{(0)}_{1,k}$ is disjoint from $R_1,R_2$ and $R_3$. 
 Thus from the congruence (\ref{c1}), we deduce the congruence (\ref{c2}) for $1\leq j_{1}<\cdots<j_{12}\leq 24$.
This completes the proof.
\end{proof}

Theorem \ref{t1} can be generalized to higher dimensions:

\begin{thm}\label{t1'}
For a positive integer $n$, we let $X\subset \P^{1}\times \P^{2n}\times \P^{2n}$ be the rank one degeneracy locus of a very general map of vector bundles
\[
\O^{\oplus (2n+1)}\rightarrow \O(2,2,0)\oplus \O(2,0,2).
\]
Then the first projection gives a family $X\rightarrow \P^{1}$ of $\O$-acyclic $2n$-folds such that any multi-section has even degree over the base $\P^{1}$.
That is, the index $I(X_{\eta})$ is even, where $X_{\eta}$ is the generic fiber.
\end{thm}
\begin{proof}
The geometry of the family of $\O$-acyclic $2n$-folds is similar to that of Lemma \ref{l1}.
An alternative projective model of $X$ is given by a complete intersection in
\[
\P^{1}\times \P_{\P^{2n}\times \P^{2n}}(\O(2,0)\oplus \O(0,2))
\]
of $(2n+1)$ divisors of type $(2,1)$,
and the intersection
\[
X\cap (\P^{1}\times \P_{\P^{2n}\times \P^{2n}}(\O(2,0)))
\]
consists of $(2n+1)2^{2n+1}$ components $E_{1,1},\ldots, E_{1,(2n+1)2^{2n+1}}$.
The theorem follows from a congruence
\begin{eqnarray*}
\deg(\alpha/\P^{1})\equiv \sum_{k=1}^{(2n+1)2^{2n}}\alpha\cdot E_{1,j_{k}} \mod 2
\end{eqnarray*}
for any $1$-cycle $\alpha$ on $X$ and for any $(2n+1)2^{2n}$-tuple $1\leq j_{1}<\cdots<j_{(2n+1)2^{2n}}\leq (2n+1)2^{2n+1}$.
We leave the details of the proof to the reader.
\end{proof}

\section{Degenerations and examples over $\Q$}\label{DEQ}
We now explain how to give examples as in Theorem \ref{T1} defined over the rational numbers. The construction is similar to the one used in the previous section, but the degeneration argument now uses Enriques fibrations defined in terms $2\times 3$-minors, rather than complete intersections of three divisors.

We will work over $\overline{\Q}$ and set 
\[\P^{1}\times \P^{2}\times \P^{2}=\Proj \overline{\Q}[s,t]\times \Proj \overline{\Q}[x_{0},x_{1},x_{2}]\times \Proj \overline{\Q}[y_{0},y_{1},y_{2}].\]
The goal is to prove the following:

\begin{thm}\label{t1Q}
Let $p_{i}, q_{i}, r_{i}, s_{i}$ $(i=0,1,2)$ be general homogeneous polynomials of tridegree $(1,2,0), (0,0,2), (2,2,0), (2,0,2)$ in variables $s,t,x_i,y_i$ defined over $\Q$.
Then there exists a prime number $p$ such that,
if $X\subset \P^{1}\times \P^{2}\times \P^{2}$ is the rank one degeneracy locus of a map of vector bundles
\[
\O^{\oplus 3}\rightarrow \O(2,2,0)\oplus\O(2,0,2)
\]
given by the matrix
\begin{eqnarray}\label{eq4}
M=\begin{pmatrix}
s p_0 + p  r_0 & (s-t)p_1+ p  r_1 & (s+t)p_2+ p  r_2 \\ 
stq_0 + p  s_0& t(s-t)q_1+ p  s_1 & t(s+t)q_2+ p  s_2 
\end{pmatrix},
\end{eqnarray}
then the first projection gives a family $X\rightarrow \P^{1}$ of Enriques surfaces such that any multi-section has even degree over the base $\P^{1}$.
That is, the index $I(X_{\eta})$ is even, where $X_{\eta}$ is the generic fiber.
\end{thm}

Note that for general $p_{i},q_{i},r_{i},s_{i}$ defined over $\Q$ and large $p$, the threefold $X$ is smooth and irreducible. 

In order to prove Theorem \ref{t1Q}, it will be convenient to introduce the following 1-dimensional family of degeneracy loci of vector bundles
\[
\O^{\oplus 3}\rightarrow \O(2,2,0)\oplus \O(2,0,2)
\]
on $\P^{1}\times\P^{2}\times \P^{2}$.
We set $B=\Proj \overline{\Q}[\lambda, \mu]$ and define the total space $\X$ as the subvariety of $B\times \PP^1\times \PP^2\times \PP^2$ defined by the maximal minors of the matrix
\begin{eqnarray}\label{eq5}
M_{(\lambda, \mu)}=\begin{pmatrix} \lambda s p_0 + \mu  r_0 & \lambda (s-t)p_1+ \mu  r_1 & \lambda (s+t)p_2+ \mu  r_2 \\ 
\lambda stq_0 + \mu  s_0& \lambda t(s-t)q_1+ \mu  s_1 & \lambda t(s+t)q_2+ \mu  s_2 \end{pmatrix},
\end{eqnarray}
where the $p_i,q_i,r_i,s_i$ have tridegrees $(1,2,0)$, $(0,0,2)$, (2,2,0) and $(2,0,2)$ respectively.

Let $\X\to B$ denote the natural projection map onto the first factor. By construction, the generic fiber $\X_{\eta_{B}}$ is a smooth threefold with an Enriques surface fibration $\X_{\eta_{B}}\rightarrow \P^{1}_{\eta_{B}}$.  
The morphism $\X\to B$ is flat outside of the fiber $(\lambda, \mu)=(1,0)$; 
we will compute the flat closure of  $\X_{\eta_{B}}$ in $B\times \PP^1\times \PP^2\times \PP^2$ below. 
In any case, in order to prove Theorem \ref{t1Q}, we will mainly be interested in the fiber over $(\lambda, \mu)=(1,p)$.

For now, let $\E_{1}\subset \X$ denote the codimension $1$ subscheme defined by the top row of $M_{(\lambda, \mu)}$, i.e.,
\begin{eqnarray}\label{eq6}
\lambda s p_0 + \mu  r_0 = \lambda (s-t)p_1+ \mu  r_1 = \lambda (s+t)p_2+ \mu r_2 = 0.
\end{eqnarray}
By Bertini, $\E_1$ is smooth and irreducible for general $p_i,r_i$. 
Let $\E_{1}\rightarrow \widetilde{B}\rightarrow B$ denote the Stein factorization of $\E_{1}\rightarrow B$.
The morphism $\widetilde{B}\to B$ is finite of degree 24; over a general point $b\in B$ the fiber corresponds to the 24 distinct planes $E_{1,1},\ldots,E_{1,24}$ in $\X_b$. 
We let $B^{\circ}\subset B$ denote the maximal open set over which $\widetilde{B}\rightarrow B$ is \'etale. 
There is an associated monodromy representation
$$
\rho : \pi_1^{\et}(B^\circ)\to S_{24}.
$$

\begin{lem}
For general $p_i,r_i$ as above, the map $\rho$ is surjective.
\end{lem}
\begin{proof}
Note that $\widetilde{B}$ is defined by \eqref{eq6} inside $B\times \PP^1\times \PP^2$. 
It is straightforward to check that the cover $\widetilde{B}\rightarrow B$ is Lefschetz for general $p_{i},r_{i}$.
Now the assertion follows from an argument similar to that in the proof of Theorem \ref{t1}.
We note that $\pi_{1}^{\et}(B^{\circ})$ is generated by loops around the branch points of $\widetilde{B}\rightarrow B$ \cite[XIII, Corollaire 2.12]{Gro} and the image of each loop is a transposition in $S_{24}$.
\end{proof}

The parameter space for the families of threefolds given by (\ref{eq5}) is a certain rational variety, hence has a Zariski dense set of $\Q$-rational points.
As a consequence, we can choose $p_{i},q_{i},r_{i},s_{i}$ defined over $\Q$ such that $\rho$ is surjective.  We will therefore in the following choose $p_i,q_i,r_i,s_i$ satisfying the above conditions: thus for the family $\X\to B$ over $\Q$, the generic fiber is smooth and irreducible; $\E_1$ is smooth and irreducible; and the monodromy map
\[
\rho \colon \pi_{1}^{\et}(B^{\circ})\rightarrow S_{24}
\]
is surjective.

\begin{lem}\label{fullmonodromy}
There are infinitely many prime numbers $p$ such that
if $x\in B$ is given by $(\lambda, \mu)=(1,p)$, then the induced map
\[\rho_{x}: \pi_1^{\et}(x,\overline{x})=\Gal(\overline{\Q}/\Q)\to S_{24}\]
 is surjective.
\end{lem}
\begin{proof}
The setting resembles that of \cite[Section 1]{Te} (but is more classical).
By Hilbert's irreducibility theorem, the set $\left\{x\in B^{\circ}(\Q)\mid \rho_{x}\text{ is surjective}\right\}$ is the complement of a thin set in $B(\Q)=\P^{1}(\Q)$.
Moreover, the complement of a thin set in $\P^{1}(\Q)$ contains infinitely many points with $(\lambda, \mu)=(1,p)$ for some prime number $p$ (see \cite[Section 9.6, Theorem]{Serre}), which gives us the desired conclusion.
\end{proof}

To conclude the proof of Theorem \ref{t1Q}, we again use a specialization argument as in Theorem \ref{t1}. 
We begin by computing the flat limit of the family $\X\to B $ over $(\lambda, \mu)=(1,0)$.

Note that $\X$ contains $\left\{\mu=t=0\right\}$ as a component. 
Removing this component reveals that the flat closure of $\X_{\eta_{B}}$ in $B\times \P^{1}\times \P^{2}\times \P^{2}$ is defined by the $3\times 3$-minors of the matrix
$$
\begin{pmatrix}
\lambda s p_0 + \mu r_0 & \lambda (s-t) p_1+ \mu r_1 & \lambda (s+t) p_2 + \mu r_2 & 0 \\
\lambda s q_0 & \lambda (s-t)q_1 & \lambda (s+t)q_2 & \mu\\
s_0 & s_1 & s_2  & -t \\
\end{pmatrix}
$$


The corresponding family $\overline{\X}\to B$ is flat and has special fiber $\overline{\X}_0$ over $(\lambda, \mu)=(1,0)$ given by a union $\widetilde{X}_{0}\cup R_{0}\cup R_{1}\cup R_{2}\cup R_{3}$, where $\widetilde{X}_{0}$ is given by the minors of the matrix
$$
N=\begin{pmatrix} p_0 & p_1 & p_2 \\ q_0 & q_1 & q_2 \end{pmatrix},
$$
$R_{0}$ is given by
$$
t=\det \begin{pmatrix}p_0 & p_1 & p_2\\ q_0 & q_1 & q_2 \\ s_0 &s_1 & s_2\end{pmatrix}=0,
$$
and $R_{1},R_{2},R_{3}$ are respectively given by
\begin{eqnarray*}
s=p_2q_1-p_1q_2=0,\, s-t =p_2q_0-p_0q_2=0,\, s+t=p_1q_0-p_0q_1=0.
\end{eqnarray*}
Note that the $R_i$ are pairwise disjoint, and $\widetilde{X}_{0}$ is regular if the $p_i,q_i$ are general. 

Similarly, the subfamily $\E_{1}\rightarrow B$, given by \eqref{eq5}, has a special fiber over $(\lambda, \mu)=(1,0)$
 which consists of the union of $12$ components $E_{1,1}^{(0)},\ldots, E_{1,12}^{(0)}$ supported on $X_{0}$ given by
\[
p_{0}=p_{1}=p_{2}=0
\]
 and the unions of $4$ components $E_{1,1}^{(l)},\ldots, E_{1,4}^{(l)}$ supported on $R_{l}$ for $l=1,2,3$ respectively given by
 \[
 s=p_{1}=p_{2}=0,\, s-t=p_{0}=p_{2}=0, \, s+t=p_{0}=p_{1}=0.
 \]
 
It is important to note that $E_{1,1}^{(0)},\ldots, E_{1,12}^{(0)}$ are Cartier divisors on $\overline{\X}_{0}$ since they are supported on $\overline{\X}_{0}\setminus (R_{0}\cup R_{1}\cup R_{2}\cup R_{3})$ which is regular.


Let $p$ and $x\in B$ be as in Lemma \ref{fullmonodromy}. 
For any valuation ring $R\subset \overline{\Q}$ whose maximal ideal contains $p$, we have the following diagram of restrictions:
\[
\xymatrix{
\X_{\overline{x}}=\overline{\X}_{\overline{x}} \ar[d] \ar[r] & (\overline{\X}_{R})_{(\lambda,\mu)=(1,p)} \ar[d] \ar[r]   &\overline{\X}_{R} \ar[d] &  \overline{\X}\ar[d]\ar[l]      \\
\overline{x} \ar[r]     & \Spec R \ar[r]^{(\lambda, \mu)=(1,p)} & B_{R} &    B\ar[l]                     \\
& \Spec \overline{\F}_{p} \ar[u] \ar[r]   & \Spec R \ar[u]_{(\lambda,\mu)=(1,0)}   &            
}.
\]

\def\M{\mathcal M}

\begin{proof}[Proof of Theorem \ref{t1Q}]
Let $p_{i},q_{i},r_{i},s_{i}$ be general and defined over $\Q$.
Let $p$ be a sufficiently large prime number which satisfies Lemma \ref{fullmonodromy} and let $X=\X_{\overline{x}}$.
We prove that any multi-section of $X\rightarrow \P^{1}$ has even degree over the base $\P^{1}$.
As in the proof of Theorem \ref{t1}, it is enough to prove, for any $1$-cycle $\alpha$ on $X$ and for any $12$-tuple $1\leq j_{1}<\cdots<j_{12}\leq 24$, a congruence
\[
\deg(\alpha/\P^{1})\equiv \sum_{k=1}^{12}\alpha\cdot E_{1,j_{k}} \mod 2.
\]
By Lemma \ref{fullmonodromy}, it suffices to verify this congruence for some $12$-tuple $1\leq j_{1}<\cdots< j_{12}\leq 24$.
To establish this, we use the above family over $\Spec R$, which allows us to specialize cycles from $X$ to cycles on $((\overline{\X}_{R})_{(\lambda, \mu) = (1,p)})_{\overline{\F}_{p}}$.

For a sufficiently large valuation ring $R\subset \overline{\Q}$ whose maximal ideal contains $p$, the specialization $((\E_{1})_{R})_{(\lambda, \mu) = (1, p)}\subset (\overline{\X}_{R})_{(\lambda,\mu) = (1,p)}$ is a disjoint union of $24$ components $\E_{1,1},\ldots, \E_{1,24}$, each of which is isomorphic to $\P^{2}_{R}$.
Let $\E_{1,j_{1}},\ldots, \E_{1,j_{12}}$ be the components which restrict to $E_{1,1}^{(0)},\cdots, E_{1,12}^{(0)}$ on the special fiber $((\overline{\X}_{R})_{(\lambda,\mu)=(1,p)})_{\overline{\F}_{p}}$.
Then $\E_{1,j_{1}},\ldots, \E_{1,j_{12}}\subset (\overline{\X}_{R})_{(\lambda, \mu) =(1,p)}$ are Cartier divisors since they are supported on $(\overline{\X}_{R})_{(\lambda, \mu) =(1, p)}\setminus (R_{0}\cup R_{1}\cup R_{2}\cup R_{3})$ which is regular.

Now by the specialization homomorphism for Chow groups \cite[Ex. 20.3.5]{F}, the desired congruence follows from 
a congruence in the proof of \cite[Theorem 3.1]{OS}:
we have, for any $1$-cycle $\widetilde{\alpha}_{0}$ on $\widetilde{X}_{0}$, a congruence
\[
\deg(\widetilde{\alpha}_{0}/\P^{1})\equiv \sum_{j=1}^{12} \widetilde{\alpha}_{0} \cdot E_{1,j}^{(0)} \mod 2.
\]
The proof is complete.
\end{proof}
The above proof uses a specialization argument which does not extend in general to other fields.
One natural question is whether one can find such examples defined over the algebraic closure of a finite field.
In contrast to the examples above, we prove some positive results in this situation, conditional on the Tate conjecture.
We recall that the Tate conjecture in degree $2i$ on a smooth projective variety $V$ over a finite field $k$ of characteristic $p$ asserts that 
the image of the cycle class map
\[
\cl^{i}\otimes \Q_{l}: CH^{i}(V_{\overline{k}})\otimes \Q_{l}\rightarrow H^{2i}_{\et}(V_{\overline{k}},\Q_{l}(i))
\]
generates the subspace of classes in $H^{2i}_{\et}(V_{\overline{k}},\Q_{l}(i))$ fixed by some open subgroup of $\Gal(\overline{k}/k)$
for any prime number $l\neq p$.
The integral Tate conjecture is an integral analogue of the Tate conjecture (with $\Z_{l}$ instead of $\Q_{l}$).

\begin{prop}\label{Tate}
Let $X$ be a smooth projective variety over $\overline{\F}_{p}$ with fibration $X\rightarrow C$ over a smooth projective curve $C$.
Assume that
\begin{enumerate}
\item the generic fiber $X_{\eta}$ is smooth with $\chi(\O_{X_{\eta}})=1$;
\item $b_{2}=\rho$ on $X$, where $b_{2}$ is the second Betti number and $\rho$ is the Neron-Severi rank;
\item the Tate conjecture holds in degree $2$ on surfaces over finite fields of characteristic $p$.
\end{enumerate}
Then the fibration $X\rightarrow C$ admits multi-sections whose degrees over the base $C$ add up to a power of $p$.
That is, the index $I(X_{\eta})$ is a power of $p$, where $X_{\eta}$ is the generic fiber.
\end{prop}
\begin{rem}
A similar assertion was proved by Colliot-Th\'el\`ene and Szamuely \cite[Theorem 6.1]{CTS}, 
where, among other things, the torsion-freeness of the Picard group $\Pic(X_{\overline{\eta}})$ of the geometric generic fiber $X_{\overline{\eta}}$ is assumed.
\end{rem}
\begin{proof}[Proof of Proposition \ref{Tate}]
Let $X\rightarrow C$ be a fibration as in the statement and $d=\dim X$.
Under the assumption (1), the Riemann-Roch formula together with the Poincar\'e duality shows that the push-forward homomorphism
\[
H^{2d-2}(X,\Z_{l}(d-1))\rightarrow H^{0}(C,\Z_{l})=\Z_{l}
\]
is surjective for any prime $l\neq p$
(the arguments are analogous to the proofs of Proposition \ref{app:prop:divisiblebyn} and Corollary \ref{app:cor:existencealpha} due to Wittenberg).

On the other hand, if $b_{2}=\rho$ on $X$, the cokernel of the cycle class map
\[
\cl^{2d-2}\otimes \Z_{l}\colon CH^{2d-2}(X)\otimes \Z_{l}\rightarrow H^{2d-2}(X,\Z_{l}(d-1))
\]
is finite by the hard Lefschetz theorem due to Deligne.
If we further assume that the Tate conjecture holds in degree $2$ on surfaces over finite fields of characteristic $p$,
then the integral Tate conjecture holds in degree $2d-2$ on $X$
(viewed as the base extension of a smooth projective variety over a finite field of characteristic $p$),
according to a theorem of Schoen \cite[Theorem 0.5]{Scho}.
This implies that the cokernel of $\cl^{2d-2}\otimes \Z_{l}$  is torsion-free, hence $\cl^{2d-2}\otimes \Z_{l}$ is surjective.
Combined with the argument in the previous paragraph, the statement now follows.
\end{proof}

\begin{cor}\label{Tate'}
Let $X\rightarrow C$ be a one-parameter family of $\O$-acyclic varieties over $\overline{\Q}$.
Assume that the Tate conjecture holds in degree $2$ on surfaces over finite fields.
Then the reduction $X_{p}\rightarrow C_{p}$ over $\overline{\F}_{p}$ admits multi-sections with coprime degrees over the base $C_{p}$ for any large prime number $p$.
That is, $I((X_{p})_{\eta})=1$, where $(X_{p})_{\eta}$ is the generic fiber.
\end{cor}
\begin{proof}
We note that the $\O$-acyclicity of fibers of the family $X\rightarrow C$ implies $b_{2}=\rho$ on $X$ by \cite[Proposition 7.3]{CTV}, 
thus we also have $b_{2}=\rho$ on the good reductions of $X$ by specialization.
Now the statement is immediate from Proposition \ref{Tate} by observing that there exist $1$-cycles on $X_{p}$ obtained by spreading out $1$-cycles on $X$ over valuations rings inside $\overline{\Q}$, whose degrees over the base $C_{p}$ do not depend on $p$.
\end{proof}

\section{Failure of the Hasse principle and the reciprocity obstruction}\label{FHPRO}

The reciprocity obstruction to the Hasse principle for $0$-cycles of degree $1$ on a smooth projective geometrically connected variety $Z$ over the function field $F=\C(C)$ of a complex curve $C$ was defined and pointed out to the authors by Colliot-Th\'el\`ene (see also \cite[Section 5]{CTG}).
We explain the construction in the following.
We will assume that $H^{1}_{\et}(Z_{\overline{F}},\Z/2)=\Z/2$ for simplicity.

The Leray spectral sequence for the \'etale sheaf $\Z/2$ and the morphism $Z\rightarrow \Spec F$ yields a short exact sequence
\[
0\rightarrow H^{1}_{\et}(F,\Z/2)\rightarrow H^{1}_{\et}(Z,\Z/2)\rightarrow H^{1}_{\et}(Z_{\overline{F}},\Z/2)\rightarrow 0.
\]
Note that the Galois group $\Gal(\overline{F}/F)$ acts trivially on $H^{1}_{\et}(Z_{\overline{F}},\Z/2)=\Z/2$.
We then choose a lift $\xi \in H^{1}_{\et}(Z,\Z/2)$ of the non-trivial class in $H^{1}_{\et}(Z_{\overline{F}},\Z/2)=\Z/2$.

The evaluation pairing 
\[
Z(F)\times H^{1}_{\et}(Z, \Z/2)\rightarrow H^{1}_{\et}(F,\Z/2)
\]
 extends to an evaluation pairing on the Chow group of $0$-cycles
\[
CH_{0}(Z)\times H^{1}_{\et}(Z,\Z/2)\rightarrow H^{1}_{\et}(F,\Z/2).
\]
Thus we get the evaluation map of $\xi$
\[
CH_{0}(Z)\rightarrow H^{1}_{\et}(F,\Z/2).
\]
Similarly, we get the local evaluation map of $\xi$
\[
CH_{0}(Z_{F_{p}})\rightarrow H^{1}_{\et}(F_{p},Z/2)=\Z/2
\]
for any $p\in C$, where $F_{p}\cong \C((t))$ is the completion of $F$ at $p$.
The local evaluation maps are identically zero for all but finitely many $p\in C$ by an argument of good reduction.

The diagonal embedding 
\[F\hookrightarrow \prod_{p\in C} F_{p}
\]
yields a complex
\[
H^{1}_{\et}(F,\Z/2)\rightarrow \bigoplus_{p\in C} H^{1}_{\et}(F_{p},\Z/2) \rightarrow \Z/2
\]
that is
\[
F^{\ast}/F^{\ast 2}\rightarrow \bigoplus_{p\in C} \Z/2 \rightarrow \Z/2,
\]
where the first map is induced by the divisor map and the second map is the summation map.
Then it follows that the image of the diagonal map
\[
CH_{0}(Z)\rightarrow \prod_{p\in C}CH_{0}(Z_{F_{p}})
\]
is contained in the kernel of the sum of the local evaluations
\[
\theta \colon \prod_{p\in C} CH_{0}(Z_{F_{p}})\rightarrow \Z/2.
\]

\begin{prop}[Reciprocity obstruction]\label{reciprocity}
If for each family $\left\{\alpha_{p}\right\}_{p\in C}$ of $0$-cycles of degree $1$, we have $\theta(\left\{\alpha_{p}\right\})=1\in \Z/2$, then there is no $0$-cycle of degree $1$ on $Z$.
\end{prop}

As a consequence of our construction in Section \ref{PMT}, we prove that the failure of the Hasse principle for $0$-cycles of degree $1$ on an Enriques surface over $\C(\P^{1})$ cannot always be accounted for by the reciprocity obstruction.

\begin{thm}\label{HPRO}
Let $X_{\eta}$ be the generic fiber of the family $X\rightarrow \P^{1}$ of Enriques surfaces of Theorem \ref{t1}.
Then the Hasse principle fails for $0$-cycles of degree $1$ on $X_{\eta}$, while the assumption of Proposition \ref{reciprocity} is not satisfied.
\end{thm}
\begin{rem}
In fact, a direct computation shows that $X_{\eta}$ has rational points everywhere locally.
Hence the Hasse principle already fails for rational points on $X_{\eta}$.
The proof in the following also shows that there is no reciprocity obstruction to the Hasse principle for rational points on $X_{\eta}$.
Therefore it follows that the reciprocity obstruction to the Hasse principle for rational points on an Enriques surface over $\C(\P^{1})$ is not the only obstruction.
\end{rem}
\begin{proof}[Proof of Theorem \ref{HPRO}]
Let $F=\C(\P^{1})$.
Theorem \ref{t1} shows that there is no $0$-cycle of degree $1$ on $X_{F}$.
On the other hand, it is automatic from the $\O$-acyclicity of Enriques surfaces and the Riemann-Roch theorem that the family $X\rightarrow \P^{1}$ has no multiple fiber (in fact this is easy to see directly from the defining equations).
It then follows from Hensel's lemma that there is a $0$-cycle of degree $1$ on $X_{F_{p}}$ for any $p\in \P^{1}$.
Therefore the Hasse principle fails for $0$-cycles of degree $1$ on $X_{F}$.

By choosing a lift $\xi \in H^{1}_{\et}(X_{F},\Z/2)$ of the non-zero class in $H^{1}_{\et}(X_{\overline{F}},\Z/2)=\Z/2$, we obtain the map $\theta$ in Proposition \ref{reciprocity}.
To see that $X_{F}$ does not satisfy the assumption of Proposition \ref{reciprocity}, it is enough to verify the following:
for each $i$ and $j$,
if $p_{i,j}\in \P^{1}$ is the image of $E_{i,j}$, then the local evaluation map
\[
CH_{0}(X_{F_{p_{i,j}}})\rightarrow H^{1}_{\et}(F_{p_{i,j}},\Z/2)=\Z/2 
\]
restricts to a surjection on $0$-cycles of degree $1$;
this will then provide a family $\left\{\alpha_{p}\right\}_{p\in C}$ of $0$-cycles of degree $1$ such that $\theta(\left\{\alpha_{p}\right\})=0 \in \Z/2$.

Recall that by construction in Section \ref{FES}, $X$ admits a natural double cover $Y\rightarrow X$ over $\P^{1}$
and the cover is ramified along $\bigcup F_{i,j}$ and branched over $\bigcup E_{i,j}$.
Then one can in fact assume that $\xi$ is given by the \'etale double cover $Y^{\circ}\rightarrow X^{\circ}$, where $X^{\circ}=X\setminus \bigcup E_{i,j}$ and $Y^{\circ}=Y\setminus \bigcup F_{i,j}$,
since evaluation maps only differ by classes in $H^{1}_{\et}(F,\Z/2)$.

For each $i$ and $j$, working locally around $p_{i,j}$,
we consider the base change of the Enriques fibration $X\rightarrow \P^{1}$
\[
X\times_{\P^{1}}\Spec \widehat{\O}_{\P^{1},p_{i,j}}\rightarrow \Spec \widehat{\O}_{\P^{1}, p_{i,j}},
\] 
where $\widehat{\O}_{\P^{1},p_{i,j}}$ is the completion of the local ring $\O_{\P^{1},p_{i,j}}$.
One can compute that the special fiber is reduced and consists of $E_{i,j}$ and a residual component $R_{i,j}$.
Then, by Hensel's lemma, there is a section $S_{1}$ (resp. $S_{2}$) which intersects transversally with $E_{i,j}$ (resp. $R_{i,j}$) at one point.
Now we consider the double cover
\[
Y\times_{\P^{1}}\Spec \widehat{\O}_{\P^{1},p_{i,j}}\rightarrow X\times_{\P^{1}}\Spec \widehat{\O}_{\P^{1},p_{i,j}},
\]
whose branched locus is $E_{i,j}$.
Then it is straightforward to see that the inverse image of $S_{1}$ gives degree $2$ integral multi-section, while that of $S_{2}$ splits into two disjoint sections.
Therefore the $F_{p_{i,j}}$-rational points of $X_{F_{p_{i,j}}}$ corresponding to the sections $S_{1}$ and $S_{2}$ 
take values $1$ and $0$ in $\Z/2$ respectively
under the local evaluation map.
This concludes that $X_{F}$ does not satisfy the assumption of Proposition \ref{reciprocity}, hence the proof of the theorem.
\end{proof}

\section{Defect of the integral Hodge conjecture in degree $4$ and degree $3$ unramified cohomology with $\Z/2$ coefficients}\label{DIHCUC}

\begin{thm}\label{t2}
Let $X$ be the total space of the family of Enriques surfaces of Theorem \ref{t1}.
Then we have $Z^{4}(X)[2]=(\Z/2)^{46}$.
\end{thm}
\begin{proof}
By Lemma \ref{l1}, the Hodge structure of $H^{4}(X,\Z)$ is trivial and $H^{4}(X,\Z)$ is free of rank $50$.
By the Tor long exact sequence, we have
\[
Z^{4}(X)[2]=\Ker(H^{4}_{\alg}(X,\Z)/2\rightarrow H^{4}(X,\Z/2)).
\]
We define 
\[
H^{4}_{\alg}(X,\Z/2)=\Image\left(\cl^{2}\otimes\Z/2\colon CH^{2}(X)/2\rightarrow H^{4}(X,\Z/2)\right).
\]
Since $H^{4}_{\alg}(X,\Z)/2=(\Z/2)^{50}$, we are reduced to proving that $H^{4}_{\alg}(X,\Z/2)=(\Z/2)^{4}$.

We first prove that
\[
\Image\left(CH^{2}(X)/2\xrightarrow{\cl\otimes \Z/2} H^{4}(X,\Z/2)\xrightarrow{(i_{X})_{*}} H^{10}(\P^{1}\times \P_{A},\Z/2)\right)=(\Z/2)^{2},
\]
where $i_{X}\colon X\rightarrow \P^{1}\times \P_{A}$ is the inclusion map.
The rank of the image is $\leq 2$ as a result of Theorem \ref{t1} and the congruence $(\ref{c3})$ in the proof.
Thus it suffices to find two linearly independent classes in the image.
It is easy to see that lines $l_{1}\subset E_{1,1}$ and $l_{2}\subset E_{2,1}$ give such classes.

We define
\[
H^{4}_{\van}(X,\Z/2)=\Ker\left((i_{X})_{*}\colon H^{4}(X,\Z/2)\rightarrow H^{10}(\P^{1}\times \P_{A},\Z/2)\right).
\]
We have $\rank H^{4}_{\van}(X,\Z/2)=46$.
Indeed, it is enough to observe that the push-forward homomorphism
\[
(i_{X})_{*}\colon H^{4}(X,\Z)\rightarrow H^{10}(\P^{1}\times \P_{A}, \Z)
\] 
is surjective, which follows from the fact that the pullback homomorphism
\[
(i_{X})^{*}\colon H^{2}(\P^{1}\times \P_{A},\Z)\rightarrow H^{2}(X,\Z)
\] 
is injective with torsion-free cokernel by Lemma \ref{l2}.
We prove that $H^{4}_{\van}(X,\Z/2)$ is generated by classes $c_{i,j_{1},j_{2}}\in H^{4}(X,\Z)$ ($i=1,2$, $1\leq j_{1}<j_{2}\leq 24$)
with intersection properties
\begin{eqnarray*}
c_{i,j_{1},j_{2}}\cdot E_{i',j'}=\delta_{i,i'}\cdot(\delta_{j_{1},j'}-\delta_{j_{2},j'}),\\
c_{i,j_{1},j_{2}}\cdot F=c_{i,j_{1},j_{2}}\cdot H_{1}=c_{i,j_{1},j_{2}}\cdot H_{2}=0.
\end{eqnarray*}
It is enough to show that 
\[H^{4}_{\van}(X,\Z)=\Ker\left((i_{X})_{*}\colon H^{4}(X,\Z)\rightarrow H^{10}(\P^{1}\times \P_{A},\Z)\right),\]
 which is of rank $46$, is generated by the above classes.
Let $i_{Y}\colon Y\rightarrow \P^{1}\times \P_{B}$ be the inclusion map
and let $$H^{4}_{\van}(Y,\Z)=\Ker\left((i_{Y})_{*}\colon H^{4}(Y,\Z)\rightarrow H^{10}(\P^{1}\times \P_{B},\Z)\right).$$
The group $H^{4}_{\van}(Y,\Z)$ has rank $46$.
Using Lemma \ref{l2},
it is straightforward to see that 
\[\Coker\left(f_{*}\colon H^{4}_{\van}(Y,\Z)\rightarrow H^{4}_{\van}(X,\Z)\right)=(\Z/2)^{46},
\] 
where $f\colon Y\rightarrow X$ is the natural map,
thus the push-forward homomorphism $f_{*}\colon H^{4}_{\van}(Y,\Z)\rightarrow H^{4}_{\van}(X,\Z)$ can be identified with the multiplication homomorphism $\Z^{46}\xrightarrow{\times 2}\Z^{46}$.
Now it is enough to observe that $H^{4}_{\van}(Y,\Z)$ is generated by classes $d_{i,j_{1},j_{2}}\in H^{4}(Y,\Z)$ ($i=1,2$, $1\leq j_{1}<j_{2}\leq 24$) with intersection properties
\begin{eqnarray*}
d_{i,j_{1},j_{2}}\cdot F_{i',j'}=\delta_{i,i'}\cdot(\delta_{j_{1},j'}-\delta_{j_{2},j'}),\\
d_{i,j_{1},j_{2}}\cdot F=d_{i,j_{1},j_{2}}\cdot H_{1}=d_{i,j_{1},j_{2}}\cdot H_{2}=0,
\end{eqnarray*}
which is immediate.

We prove that $H^{4}_{\alg}(X,\Z/2)\cap H_{\van}^{4}(X,\Z/2)=(\Z/2)^{2}$.
We note that we have the congruence (\ref{c3}) in the proof of Theorem \ref{t1}, and moreover, 
we may also assume a congruence
\begin{eqnarray}\label{c4}
\alpha\cdot E_{2,1}\equiv \cdots \equiv \alpha\cdot E_{2,24} \mod 2
\end{eqnarray}
for any $1$-cycle $\alpha$ on $X$.
Then the congruences (\ref{c3}) and (\ref{c4}) imply
\[\rank\left(H^{4}_{\alg}(X,\Z/2)\cap H_{\van}^{4}(X,\Z/2)\right)\leq 2.\]
Now it is enough to find two linearly independent classes in $H^{4}_{\alg}(X,\Z/2)\cap H_{\van}^{4}(X,\Z/2)$.
It is a simple matter to check that $C_{1}=(H_{1})^{2}$ and $C_{2}=(H_{2})^{2}$ indeed give such classes.

It follows that 
\[H^{4}_{\alg}(X,\Z/2)=\Z/2[l_{1}]\oplus \Z/2[l_{2}]\oplus \Z/2[C_{1}]\oplus \Z/2[C_{2}]=(\Z/2)^{4}.\]
The proof is complete.
\end{proof}

Let $n$ be a positive integer.
We recall that the degree $3$ unramified cohomology group $H^{3}_{\nr}(X,\Z/n)$ for a smooth projective variety $X$ is defined to be
\[
H^{3}_{\nr}(X,\Z/n)=H^{0}(X_{\Zar}, \mathcal{H}^{3}(\Z/n)),
\]
where $\mathcal{H}^{3}(\Z/n)$ is the Zariski sheaf associated to the presheaf $U\mapsto H^{3}(U,\Z/n)$ \cite{BlO}.
The group $H^{3}_{\nr}(X,\Z/n)$ is a stable birational invariant of smooth projective varieties
 \cite[Theorem 4.2]{BlO}.
As an application of the Bloch-Kato conjecture settled by Voevodsky,
it was proved by Colliot-Th\'el\`ene and Voisin \cite[Theorem 3.9]{CTV} that we have
\[
H^{3}_{\nr}(X,\Z/n)=Z^{4}(X)[n]
\]
if $CH_{0}(X)$ is supported on a surface.
This theorem, together with Lemma \ref{l1} (1) and Theorem \ref{t2}, implies:

\begin{cor}\label{co1}
Let $X$ be the total space of the family of Enriques surfaces of Theorem \ref{t1}.
Then we have $H^{3}_{\nr}(X,\Z/2)=(\Z/2)^{46}$.
\end{cor}

\appendix
\section{Vanishing of the reciprocity obstruction\texorpdfstring{\\\smallskip}{ (}by Olivier Wittenberg\texorpdfstring{}{)}}

\renewcommand{\et}{\textup{\'{e}t}}

In this appendix, we prove that 
the vanishing of the reciprocity obstruction
to the existence of a $0$\nobreakdash-cycle of degree~$1$ is a general fact that holds
for all $\O$\nobreakdash-acyclic varieties
over the function field~$F$
of a complex curve, and, in fact, for all smooth proper varieties~$Y$ over~$F$
such that $\chi(Y,\O_Y)=1$.
We actually prove the following slightly more general statement, in the spirit of \cite{ELW}.

\begin{thm}
\label{app:th:main}
Let $F=\C(C)$ be the function field of a smooth proper irreducible complex curve~$C$.
Let~$Y$ be a smooth proper variety over~$F$ and~$E$ be a coherent sheaf on~$Y$.
Then there exists a collection
$(\alpha_p)_{p \in C(\C)} \in \prod_{p \in C(\C)} \CH_0(Y_{F_p})$
of local $0$\nobreakdash-cycle classes of degree~$\chi(Y,E)$
that belongs to the left kernel of the natural pairing
\begin{align}
\label{app:pairingCH0}
\Bigg(\mkern2mu\prod_{p \in C(\C)} \CH_0(Y_{F_p})\Bigg) \times H^1_{\et}(Y,\Q/\Z(1)) \to \Q/\Z\rlap{.}
\end{align}
In other words,
there is no reciprocity obstruction to the existence of a $0$\nobreakdash-cycle of degree $\chi(Y,E)$
on~$Y$.
\end{thm}

Theorem~\ref{app:th:main} builds on a purely cohomological reinterpretation
of the reciprocity obstruction
(presented in
\textsection\textsection\ref{app:subsecrecollections}--\ref{app:subseccohreinterpretation})
and on a variant of
an argument of
Colliot-Th\'el\`ene and Voisin
itself based
on the Riemann--Roch theorem
(see~\textsection\ref{app:subsecrr}).

In the situation of Theorem~\ref{app:th:main},
local $0$\nobreakdash-cycles of degree $\chi(Y,E)$ had previously been shown to exist
in \cite[Theorem~1]{ELW}.  It may seem surprising that
the existence of a collection of local
$0$\nobreakdash-cycles that globally survives the reciprocity obstruction
comes ``for free'', without having to make any additional assumption on~$Y$,
especially in view of the negative answer to Question~\ref{q} now provided by
Ottem and Suzuki.

\subsection{Recollections on the reciprocity obstruction}
\label{app:subsecrecollections}

Let us first recall how the pairing~\eqref{app:pairingCH0},
introduced by Colliot-Th\'el\`ene and Gille in~\cite[\textsection5]{CTG},
is defined.

For $p \in C(\C)$, the Galois cohomology
group $H^1(F_p,\Q/\Z(1))$, where $\Q/\Z(1)$ denotes the torsion subgroup of~$\C^*$,
is canonically isomorphic to $\Q/\Z$.
We denote
this canonical isomorphism by $\inv_p:H^1(F_p,\Q/\Z(1))\isoto \Q/\Z$.
Mapping a closed point $q \in Y_{F_p}$ and a class $\beta \in H^1_{\et}(Y_{F_p},\Q/\Z(1))$
to $\inv_{p}\Cores_{F_p(q)/F_p}\beta(q) \in \Q/\Z$,
where $\Cores$ denotes the corestriction map in Galois cohomology,
uniquely extends to a bilinear pairing
\begin{align}
\CH_0(Y_{F_p}) \times H^1_{\et}(Y_{F_p},\Q/\Z(1)) \to \Q/\Z\rlap{.}
\end{align}
Denoting the latter by angle brackets,
the pairing~\eqref{app:pairingCH0} is then defined as the sum
\begin{align}
((\alpha_p)_{p \in C(\C)},\beta)\mapsto \sum_{p\in C(\C)} \langle \alpha_p, \beta \rangle\rlap{,}
\end{align}
which can be checked to have only finitely many non-zero terms.

The ``reciprocity law'', in this context, is the equality
\begin{align}
\label{app:reciprocitylaw}
\sum_{p \in C(\C)} \inv_{p} \gamma=0\rlap{,}
\end{align}
valid for any global class $\gamma \in H^1(F,\Q/\Z(1))$, and which amounts to the
assertion that any principal divisor on~$C$ has degree~$0$.
Applied to $\gamma=\Cores_{F(q)/F} \beta(q)$ for a closed point $q \in Y$,
it implies that the diagonal map $\CH_0(Y) \to \prod_{p \in C(\C)} \CH_0(Y_{F_p})$
takes values in the left kernel of~\eqref{app:pairingCH0}.
Equivalently,
an
element of $\prod_{p \in C(\C)} \CH_0(Y_{F_p})$
that does not belong to the left kernel of~\eqref{app:pairingCH0}
cannot come from $\CH_0(Y)$;
in this situation one says that there is a ``reciprocity obstruction''.

\subsection{From Chow groups to cohomology}

Let us fix a smooth and proper variety~$X$ over~$\C$ and a morphism $f:X \to C$
with generic fibre~$Y$.
Let~$X_p$ denote the fibre
of~$X$ above $p \in C$.
For any scheme~$Z$ of finite type over~$\C$, over~$F_p$ or over~$\widehat{\O}_{C,p}$,
and all integers $q,j$, we set
$H^{q}_{\et}(Z, \hat \Z(j))=\varprojlim_{n\geq 1} H^{q}_{\et}(Z, \Z/n\Z(j))$.
Let $d=\dim(Y)$ and $\Z(d)=(\sqrt{-1})^d\Z$.
Combining
the inverse of the isomorphism
$H^{2d}_{\et}(X \times_C \Spec(\widehat{\O}_{C,p}), \hat \Z(d))
\isoto H^{2d}_{\et}(X_p,\hat\Z(d)) $
given by the proper base change theorem
with
the canonical identification between singular and \'etale cohomology
$H^{2d}(X_p(\C),\Z(d)) \otimes_{\Z} \hat\Z = H^{2d}_{\et}(X_p,\hat\Z(d))$,
we obtain a canonical injection
\begin{align}
\label{app:eq:Zstructure}
H^{2d}(X_p(\C),\Z(d))
\hookrightarrow H^{2d}_{\et}(X \times_C \Spec(\widehat{\O}_{C,p}), \hat \Z(d))\rlap{.}
\end{align}
We shall consider the pull-back map
\begin{align}
\label{app:eq:pullbackmap}
H^{2d}_{\et}(X \times_C \Spec(\widehat{\O}_{C,p}), \hat \Z(d))\to
H^{2d}_{\et}(Y_{F_p}, \hat \Z(d))
\end{align}
and its composition
\begin{align}
\label{app:eq:pullbackmapZ}
H^{2d}(X_p(\C),\Z(d))\to
H^{2d}_{\et}(Y_{F_p}, \hat \Z(d))
\end{align}
with this injection.

\begin{prop}
\label{app:prop:imagecl}
For any $p \in C(\C)$, the image of the cycle class map to \'etale cohomology
$\cl:\CH_0(Y_{F_p}) \to H^{2d}_\et(Y_{F_p},\hat{\Z}(d))$
is equal to the image of~\eqref{app:eq:pullbackmapZ}.
\end{prop}

\begin{proof}
Let $X_{p,1},\dots,X_{p,n}$ denote the irreducible components of~$X_p$, endowed with
the reduced subscheme structure.
Let~$Z_1^h$ be the group of horizontal $1$\nobreakdash-cycles
on the scheme $X \times_C \Spec(\widehat{\O}_{C,p})$,
that is, the group of those $1$\nobreakdash-cycles whose support is flat over~$\widehat{\O}_{C,p}$.
Given $z \in Z_1^h$, let $(z\cdot X_{p,i})$ denote the intersection number
of~$z$ with the Cartier divisor $X_{p,i} \subset X_p$.
The map $Z_1^h \to \Z^n$, $z \mapsto ((z\cdot X_{p,1}),\dots,(z\cdot X_{p,n}))$
is surjective as a consequence of \cite[9.1/9]{BLR},
and
fits into a commutative diagram
\begin{align}
\begin{aligned}
\xymatrix@R=4ex@C=2em{
\Z^n &
Z_1^h
\ar@{->>}[l]
\ar[d] \ar@{->>}[r] & \CH_0(Y_{F_p}) \ar[d]^(.42)\cl \\
\ar@{=}[u]
H^{2d}(X_p(\C),\Z(d)) & \ar@{<-_{)}}[l]!(16,0)
H^{2d}_{\et}(X \times_C \Spec(\widehat{\O}_{C,p}), \hat \Z(d)) \ar[r] & H^{2d}_{\et}(Y_{F_p},\hat\Z(d))\rlap,
}
\end{aligned}
\end{align}
whose middle vertical arrow is the cycle class map (see \cite[(1.12)]{SaitoSato} for its definition),
whose lower horizontal arrows are the injection~\eqref{app:eq:Zstructure}
and the pull-back map~\eqref{app:eq:pullbackmap}, and whose leftmost vertical map
comes from the canonical isomorphisms $H^{2d}(X_{p,i}(\C),\Z(d))=\Z$ for $i \in \{1,\dots,n\}$
and from the decomposition $H^{2d}(X_p(\C),\Z(d))=\bigoplus_{i=1}^n H^{2d}(X_{p,i}(\C),\Z(d))$.
The desired statement now follows from the diagram.
\end{proof}

\subsection{Extension to a pairing between cohomology classes}

Let us denote by
$\prod'_{p \in C(\C)} H^{2d}_{\et}(Y_{F_p}, \hat \Z(d))$
the subgroup
of
$\prod_{p \in C(\C)} H^{2d}_{\et}(Y_{F_p}, \hat \Z(d))$
consisting of those families $(\alpha_p)_{p\in C(\C)}$ such that for all but finitely many $p\in C(\C)$,
the class~$\alpha_p$ belongs to the image of the pull-back map~\eqref{app:eq:pullbackmap}.
Letting~$f$ also stand for the
morphisms
$Y \to \Spec(F)$ and $Y_{F_p} \to \Spec(F_p)$ obtained from $f:X\to C$ by base change,
we consider the pairing
\begin{align}
\label{app:cohpairing}
\Bigg(\prod_{p \in C(\C)}{\rlap{\raise 5pt\hbox{$\mkern-13mu{}'$}}} H^{2d}_{\et}(Y_{F_p}, \hat \Z(d))\Bigg) \times H^1_{\et}(Y,\Q/\Z(1))
\to \Q/\Z
\end{align}
defined by $((\alpha_p)_{p \in C(\C)},\beta) \mapsto
\sum_{p \in C(\C)} \inv_p f_*(\alpha_p \smile \beta)$,
where 
\begin{align}
f_*:H^{2d+1}_{\et}(Y_{F_p},\Q/\Z(d+1))\to H^1(F_p,\Q/\Z(1))
\end{align}
is induced by the trace morphism associated with~$f$ (see \cite[XVIII, (2.13.2)]{SGA43}).
This pairing is well-defined as $\alpha_p\smile \beta$ vanishes for any~$p$ such that both~$\alpha_p$
and the image of~$\beta$
in $H^1_{\et}(Y_{F_p},\Q/\Z(1))$
come from the cohomology of $X \times_C \Spec(\widehat{\O}_{C,p})$, by the proper base change theorem.

\begin{prop}
\label{app:prop:samepairing}
The pairing
$\big(\mkern2mu\prod_{p \in C(\C)} \CH_0(Y_{F_p})\big) \times H^1_{\et}(Y,\Q/\Z(1)) \to \Q/\Z$
induced by~\eqref{app:cohpairing} via the maps
$\cl:\CH_0(Y_{F_p}) \to H^{2d}_\et(Y_{F_p},\hat{\Z}(d))$
(whose product takes values
in $\prod'_{p \in C(\C)} H^{2d}_{\et}(Y_{F_p}, \hat \Z(d))$
by Proposition~\ref{app:prop:imagecl})
is equal to the pairing~\eqref{app:pairingCH0}.
\end{prop}

\begin{proof}
Letting $i:\Spec(F_p(q)) \to Y_{F_p}$ denote the inclusion of a closed point~$q$
of~$Y_{F_p}$
and~$1$ the unit of $H^0(F_p(q),\hat{\Z})$,
we have $\cl(q)=i_*1$.
For $\beta \in H^1(Y,\Q/\Z(1))$,
the equality
$\Cores_{F_p(q)/F_p} \beta(q) = (f \circ i)_* \beta(q) = f_*i_*i^*\beta$
and the projection formula $i_*i^*\beta=i_*1 \smile \beta$
therefore yield the desired compatibility.
\end{proof}

\subsection{Cohomological reinterpretation}
\label{app:subseccohreinterpretation}

We are now in a position to reformulate the conclusion of Theorem~\ref{app:th:main} in purely cohomological terms.

\begin{prop}
\label{app:prop:cohreformulation}
For $m \in \Z$, the following conditions are equivalent:
\begin{enumerate}
\item[$(1)$]
There exists a collection
$(\alpha_p)_{p \in C(\C)} \in \prod_{p \in C(\C)} \CH_0(Y_{F_p})$
belonging to the left kernel of the pairing~\eqref{app:pairingCH0},
with $\deg(\alpha_p)=m$ for all $p\in C(\C)$.
\item[$(2)$]
There exists $\alpha \in H^{2d}(X(\C),\Z(d))$ with $f_*\alpha=m$ in $H^0(C(\C),\Z)=\Z$.
\end{enumerate}
\end{prop}

\begin{proof}
Let us consider the following variants of~(2):
\begin{enumerate}
\item[$(2')$]
There exists $\alpha \in H^{2d}_{\et}(X,\hat\Z(d))$ with $f_*\alpha=m$ in $H^0_{\et}(C,\hat\Z)=\hat\Z$.
\item[$(2'')$]
Same as $(2')$, except that we impose, in addition, that 
the image of~$\alpha$
 in $H^{2d}_{\et}(Y_{F_p},\hat\Z(d))$ lies in the image of~\eqref{app:eq:pullbackmapZ}
for all $p\in C(\C)$.
\end{enumerate}
It is clear that $(2)\Rightarrow(2'')\Rightarrow(2')$.
On the other hand, we have $(2')\Rightarrow(2)$
as
the maps $f_*:H^{2d}_{\et}(X,\hat\Z(d)) \to H^0_{\et}(C,\hat\Z)$
and $f_*:H^{2d}(X(\C),\Z(d)) \to H^0(C(\C),\Z)$
share the same cokernel,
by the comparison between singular and \'etale cohomology.
Hence $(2)$, $(2')$ and $(2'')$ are equivalent.
Now~$(2'')$ is in turn equivalent to the existence of
$\alpha\in\varinjlim H^{2d}_{\et}(X_U,\hat\Z(d))$,
where~$U$ ranges over the dense open subsets of~$C$ and~$X_U$ denotes the inverse image of~$U$ in~$X$,
whose image~$\alpha_p$ in $H^{2d}_{\et}(Y_{F_p},\hat\Z(d))$,
for all $p \in C(\C)$,
belongs to the image of~\eqref{app:eq:pullbackmapZ}
and satisfies $f_*\alpha_p=m \in \hat\Z$.
Finally, this condition is equivalent to~$(1)$
in view of the next lemma, of Proposition~\ref{app:prop:imagecl}
and of Proposition~\ref{app:prop:samepairing}.
\end{proof}

\begin{lem}
The left kernel of the pairing~\eqref{app:cohpairing}
coincides with
the image of the diagonal map
$\varinjlim H^{2d}_{\et}(X_U,\hat \Z(d))\to \prod'_{p \in C(\C)} H^{2d}_{\et}(Y_{F_p}, \hat \Z(d))$,
where~$U$ ranges over the dense open subsets of~$C$.
\end{lem}

\begin{proof}
For any bounded complex $\mathcal C$
of constructible \'etale sheaves of abelian groups on~$C$
and for any small enough dense open subset~$U$ of~$C$,
one obtains, by proceeding exactly as in the proof of
\cite[Proposition~2.6]{Izq}, an exact sequence
\begin{align}
H^0_{\et}(U,\mathcal C) \to
\prod_{p \in C(\C) \setminus U(\C)} H^0_{\et}(F_p,\mathcal C)
\times
\prod_{p \in U(\C)} H^0_{\et}(\widehat{\O}_{C,p},\mathcal C)
\to
H^1_{\et}(F,\mathcal C')^D\rlap{,}
\end{align}
where $\mathcal C'=R\mathcal{H}\mkern-1muom(\mathcal C,\Q/\Z(1))$,
where~$D$ denotes the Pontrjagin dual,
 and where we still denote by~$\mathcal C$ (resp.\ $\mathcal C'$) the pull-back of~$\mathcal C$
(resp.\ $\mathcal C'$) to any of~$U$,
$\Spec(F)$, $\Spec(\widehat{O}_{C,p})$, $\Spec(F_p)$.
The lemma now follows by considering the exact sequences associated in this way
with
$\mathcal C=Rf_*\Z/n\Z(d)[2d]$
for $n \geq 1$
and applying $\varinjlim_U \varprojlim_n$ to these sequences,
in view of the canonical isomorphism $\mathcal C'=Rf_*\Z/n\Z(1)$ given by Poincar\'e duality;
as the groups $H^0_{\et}(U,\mathcal C)=H^{2d}_{\et}(X_U,\Z/n\Z(d))$ are all finite,
the resulting sequence is still exact (Mittag--Leffler criterion).
\end{proof}

\subsection{Applying the Riemann--Roch theorem}
\label{app:subsecrr}

The next statement and its proof are a variation on
a result of Colliot-Th\'el\`ene and Voisin \cite[Proposition~7.3~(ii)]{CTV},
in the style of \cite{ELW}.
When $E=\O_Y$, 
its formulation is parallel to \cite[Proposition~2.4]{ELW}.

\begin{prop}
\label{app:prop:divisiblebyn}
Let~$E$ be a coherent sheaf on~$Y$
and~$n \geq 1$ be an integer.
If the class of the fibres of $f:X\to C$
in $\NS(X)/(\mathrm{torsion})$
is divisible by~$n$,
then $\chi(Y,E)$ is divisible by~$n$.
\end{prop}

\begin{proof}
As~$E$ is the restriction of a coherent sheaf on~$X$ and as any coherent
sheaf on~$X$ admits a finite resolution by
locally free sheaves, we may assume that~$E$ is the restriction
of a locally free sheaf~$V$ on~$X$, which we henceforth fix.

Let~$F$ be a fibre of~$f$.
By assumption, there exist divisors~$H$ and~$M$ on~$X$ such that the equality
\begin{align}
\label{app:eq:nhmf}
n[H] = [M] + [F]
\end{align}
holds in $\NS(X)$ and such that $[M]$ belongs to the torsion
subgroup of $\NS(X)$.
By the last condition, the divisor~$M$ is numerically trivial. Hence so is the cycle class $H^2 \in \CH^2(X)$,
in view of~\eqref{app:eq:nhmf} and of the equality $F^2=0 \in \CH^2(X)$.

The Hirzebruch--Riemann--Roch theorem applied
to the locally free sheaves~$V$
and $V \otimes_{\O_X} \O_X(H)$ on~$X$ therefore gives us the equality
\begin{align*}
\chi(X,V \otimes_{\O_X} \O_X(H)) - \chi(X,V) &=
\deg(\mathrm{ch}(V) \cdot (H + H^2/2 + \cdots) \cdot \Td(T_X))\\
&= \deg(H\cdot Z)
= \frac{1}{n}\deg(F\cdot Z)\rlap{,}
\end{align*}
where $Z \in \CH_1(X)\otimes_\Z \Q$ denotes the $1$\nobreakdash-dimensional
component of~$\mathrm{ch}(V)\cdot \Td(T_X)$.
By the Hirzebruch--Riemann--Roch theorem applied to the locally free sheaf~$E$ on~$Y$,
we also have the equality $\chi(Y,E)=\deg(\mathrm{ch}(E)\cdot \Td(T_Y))$,
which can be rewritten as $\chi(Y,E)=\deg(F\cdot Z)$
since $\Td(T_X)|_Y = \Td(T_Y)$ and $\mathrm{ch}(V)|_Y=\mathrm{ch}(E)$.
Hence
\begin{align}
\chi(X,V \otimes_{\O_X} \O_X(H)) - \chi(X,V) =
\frac{1}{n}\chi(Y,E)
\end{align}
and we conclude that~$n$ divides $\chi(Y,E)$ since the left-hand side is an integer.
\end{proof}

\begin{cor}
\label{app:cor:existencealpha}
For any coherent sheaf~$E$ on~$Y$, there exists $\alpha \in H^{2d}(X(\C),{\Z}(d))$
such that $f_*\alpha=\chi(Y,E)$ in $H^{0}(C(\C),\Z)=\Z$.
\end{cor}

\begin{proof}
According to Proposition~\ref{app:prop:divisiblebyn},
the integer $\chi(Y,E)$
annihilates
the kernel of
$f^*:\NS(C) \otimes_\Z \Q/\Z \to \NS(X) \otimes_\Z \Q/\Z$.
As
$\NS(C) \otimes_\Z \Q/\Z = H^2(C(\C),\Q/\Z(1))$
and
$\NS(X) \otimes_\Z \Q/\Z \subset H^2(X(\C),\Q/\Z(1))$,
the latter kernel coincides with the kernel of
$f^*:H^2(C(\C),\Q/\Z(1)) \to H^2(X(\C),\Q/\Z(1))$.
Thus, Poincar\'e duality implies that  $\chi(Y,E)$ also annihilates the cokernel of
$f_*:H^{2d}(X(\C),{\Z}(d)) \to H^0(C(\C),{\Z})$.
\end{proof}

Combining Proposition~\ref{app:prop:cohreformulation}
with Corollary~\ref{app:cor:existencealpha} now yields
Theorem~\ref{app:th:main}.

\end{document}